\documentclass{amsart}
 \usepackage{graphicx}
\usepackage{amssymb, color}\usepackage{amsmath}
\usepackage{amsfonts}
\usepackage{latexsym}
\makeatletter
 \def\LaTeX{\leavevmode L\raise.42ex
   \hbox{\kern-.3em\size{\sf@size}{0pt}\selectfont A}\kern-.15em\TeX}
\makeatother

\newcommand{\BibTeX}{{\rm B\kern-.05em{\sc
i\kern-.025emb}\kern-.08em\TeX}}
\newtheorem{col}{Corollary}[section]
\newtheorem{thm}{Theorem}[section]
\newtheorem{defn}{Definition}

\theoremstyle{defn}
\newtheorem{lem}[thm]{Lemma}
\newtheorem{exmp}{Example}
\newtheorem{remark}[thm]{Remark}

\numberwithin{equation}{section}

\begin{document}

\title[Jackson-type inequality]{ Jackson-type inequality  in Hilbert spaces and on homogeneous  manifolds}
\author{Isaac Z. Pesenson}
\address{Department of Mathematics, Temple University,
Philadelphia, PA 19122} \email{pesenson@temple.edu}

\keywords{Jackson-type inequality, K-functor, one-parameter groups  of operators,  Paley-Wiener vectors, modulus of continuity, unitary representations of Lie groups, homogeneous manifolds}
  \subjclass{ 43A85, 41A17;}

 \begin{abstract} We consider a Hilbert  space ${\bf H}$ equipped with a set of strongly continuous bounded semigroups satisfying certain conditions. 
 The conditions allow  to  define a family of moduli of continuity $\Omega^{r}(s,f),\>r\in \mathbb{N}, s>0,$ of vectors in ${\bf H}$ and a family of Paley-Wiener subspaces $PW_{\sigma}$ parametrized by bandwidth $\sigma>0$.
 These subspaces are explored   to  introduce notion of the best approximation $\mathcal{E}(\sigma, f)$ of a general vector in ${\bf H}$ by Paley-Wiener vectors of a certain bandwidth $\sigma>0$. The main objective of the paper is to prove the so-called Jackson-type estimate $\mathcal{E}(\sigma, f)\leq C\left( \Omega^{r}(\sigma^{-1},f)+\sigma^{-r}\|f\|\right)$ for  $\sigma>1$.  
Our assumptions are satisfied for  a strongly continuous unitary representation of a Lie group $G$ in a Hilbert space ${\bf H}$.  It allows  to obtain the Jackson-type estimates on homogeneous manifolds.

\end{abstract}

\maketitle

 \section{Introduction and Main Results}

One of the main goals  of the classical harmonic analysis is to describe relations between frequency content of a function and its smoothness.   A famous result in this direction is the so-called Jackson Theorem for functions in $L_{2}(\mathbb{R})$: 
\begin{equation}\label{classical}
\inf_{g\in
PW_{\sigma}(\mathbb{R})}\|f-g\|=\|f-\mathcal{P}_{\sigma}f\|=\mathcal{E}(\sigma, f)\leq C\omega^{r}(\sigma^{-1}, f),\>\>\sigma>0,\>\>r\in \mathbb{N},
\end{equation}
where the Paley-Wiener space $PW_{\sigma}(\mathbb{R}), \>\sigma>0,$ is the space of functions in $L_{2}(\mathbb{R})$ whose Fourier transform has support in $[-\sigma,\>\sigma]$, $\mathcal{P}$ is the orthogonal projection of $L_{2}(\mathbb{R})$ onto $PW_{\sigma}(\mathbb{R})$, and 
$$
\omega^{r}( s, f)=\sup_{0\leq \tau\leq s}\|(T(\tau)-I)^{r}f\|_{2}, \>\>\>\>\>T(\tau)f(\cdot)=f(\cdot+\tau),
$$
 is the modulus of continuity. Similar estimate also holds true in the case of one dimensional torus $\mathbb{T}$ if one will replace $PW_{\sigma}$ by the space of trigonometric polynomials degree $\leq n$. In the case of $\mathbb{R}^{d}$ and $\mathbb{T}^{d}$ one defines corresponding modules of continuity by using one-parameter translation groups along single coordinates
\begin{equation}\label{transl}
T_{j}(\tau)f(x_{1}, ..., x_{j} ,... , x_{d})= f(x_{1}, ..., x_{j}+\tau ,... ,x_{d}),\>\>\>f\in L_{p},\>1\leq p\leq \infty,
\end{equation}
whose infinitesimal operators are partial derivatives  $\partial/\partial x_{j},\>1\leq j\leq d$.
The corresponding spaces $PW_{\sigma}$ and spaces on trigonometric polynomials can be introduced in terms of the Laplace operator $\Delta=\partial_{1}^{2}+...+\partial_{d}^{2}$.

The main objective of this paper is to develop a unified approach to Jackson-type estimates in a Hilbert space in which space  in which a family of strongly continuous bounded semigroups is given. It allows us to  consider an appropriate notion  of  Paley-Wiener vectors and  a modulus of continuity in a Hilbert space $\mathbf{H}$ space of unitary representation of a Lie group $G$ (see definitions below) and to prove an analog of the Jackson inequality (\ref{classical}) in a such general setting. We  apply these results to function spaces on homogeneous manifolds, i.e. to manifolds which have many symmetries. Our development is extensively using the notion of Peetre's K-functional.

Note that an approach to a generalization of the classical approximation theory and K-functional to abstract spaces in which a strongly continuous bounded representation of a Lie group is given was outlined  without complete proofs in \cite{Pes1}-\cite{Pes6}. The problem of developing approximation theory and K-functional in non-classical  settings  attracted attention of many mathematicians and in particular was treated in \cite{D, DX1, DX2, D1, D2, FFP,  KR, NRT1, NRT2, O,  Pes8}.

\bigskip

We consider a Hilbert space $\mathbf{H}$ and operators $D_{1}, D_{2},...,D_{d}$ which generate strongly continuous uniformly bounded semigroups $T_{1}(t), T_{2}(t),...,T_{d}(t), \>\> \|T(t)\|\leq 1, \>\>t\geq 0,$  (see \cite{BB} for the general theory of the one-parameter semigroups).  An analog of a Sobolev space is introduced as the space $\mathbf{H}^{r}$ of vectors in $\mathbf{H} $ for which the following norm is finite
$$
|||f|||_{\mathbf{H}^{r}}=\|f\|_{\mathbf{H}}+\sum_{k=1}^{r}\sum_{1\leq j_{1},  ...j_{k}\leq d}\|D_{j_{1}}...D_{j_{k}}f\|_{\mathbf{H}},
$$
 where $r\in \mathbb{N},\>\> f\in \mathbf{H}.$
  By using the closed graph theorem and
the fact that each $D_{i}$ is a closed operator in $\mathbf{H}$,\ one can show that this norm
 is equivalent to the norm
\begin{equation}\label{Sob-0}
\|f\|_{r}=\|f\|_{\mathbf{H}}+\sum_{1\leq i_{1},..., i_{r}\leq
d}\|D_{i_{1}}...D_{i_{k}}f\|_{\mathbf{H}},\>\> ~ r\in \mathbb{N}.
\end{equation}
  Let $\mathcal{D}(D_{i})$ be the domain of the operator $D_{i}$.
For every $f\in \mathbf{H}$ we introduce a vector-valued function
$
Tf: \mathbb{R}^{d}\longmapsto \mathbf{H}
$
defined as
$
Tf(t_{1}, t_{2},  ..., t_{d})=T_{1}(t_{1}) T_{2}(t_{2})  ... T_{d}(t_{d})f.
$

{\bf Assumptions.}

\textit{ Our main assumption is that we consider a Hilbert space $\mathbf{H}$ and operators $D_{1}, D_{2},...,D_{d}$ which generate strongly continuous uniformly bounded semigroups $T_{1}(t), T_{2}(t),...,T_{d}(t), \>\> \|T(t)\|\leq 1, \>\>t\geq 0,$ such that the following properties hold:}

\bigskip

 \textit{(a) There exists a set $\mathcal{G}\subset \mathbf{H}^{1}= \bigcap_{i=1}^{d}\mathcal{D}(D_{i})$ which  is dense in $\mathbf{H}$ and  invariant with respect to all $T_{i}(t), \>\>1\leq i\leq d,\>\>t\geq 0.$ }

 \textit{(b)  For every $1\leq i\leq d,$ every $f\in \mathcal{G}$ and all $\mathbf{t}=(t_{1}, ..., t_{d})$ in the standard open unit ball $U$ in $\mathbb{R}^{d}$}

\begin{equation}\label{assumpt-1}
D_{i}Tf(t_{1}, ..., t_{d})=\sum_{k=1}^{d}\zeta^{k}_{i}(\mathbf{t})\left(\partial_{k} Tf\right)(t_{1}, ..., t_{d}),
\end{equation}
\textit{where $\zeta^{k}_{i}(\mathbf{t})$ belong to $C^{\infty}(U),\>\partial_{k}=\frac{\partial}{\partial t_{k}}.$}

 \textit{(c)  The operator $L=-D_{1}^{2}-... -D_{d}^{2}$
is a non-negative self-adjoint operator in $\mathbf{H}$ and the domain of $L^{r/2},\>r\in \mathbb{N},$ with the graph norm $\|f\|+\|L^{r/2}f\|$ coincides with the space $\mathbf{H}^{r}$ with the norm (1.3).}

Using the groups  $T_{1},..., T_{d},\>d\geq n=dim\>M,$  we define an analog of the  modulus of continuity  by the formula
\begin{equation}\label{Mod-0}
\Omega^{r}( s, f)=
 $$
 $$
 \sum_{1\leq j_{1},...,j_{r}\leq
d}\sup_{0\leq\tau_{j_{1}}\leq s}...\sup_{0\leq\tau_{j_{r}}\leq
s}   \|
\left(T_{j_{1}}(\tau_{j_{1}})-I\right)
...\left(T_{j_{r}}(\tau_{j_{r}})-I\right)f \|_{{\bf H}},
\label{M}
\end{equation}
where $\>f\in {\bf H}, 
\> r\in \mathbb{N},  $ and $I$ is the
identity operator in ${\bf H}.$

The following statement was proved in \cite{Pes1}, \cite{Pes4}, \cite{Pes8}: there exist positive  constants $c_{1},\>C_{1}$ such that for every $f\in \mathbf{H}$
\begin{equation}\label{double-ineq-0}
 c_{1}\Omega^{r}(s,f) \leq K(s^{r},f, {\bf H}, {\bf H}^{r})   \leq
C_{1}\left(\Omega^{r}(s,f)+ \min(s^{r}, 1)\|f\|_{{\bf H}}\right),
\end{equation}
where 
\begin{equation}\label{KfuncH}
K(s^{r},f, {\bf H}, {\bf H}^{r})=\inf_{g\in {\bf H}^{r}}\left(  \|f-g\|_{{\bf H}}+s^{r}\|g\|_{{\bf H}^{r}}\right).
\end{equation}
The property \textit{(d)} on the  list of properties allows us to utilize  the operator $L=-D_{1}^{2}-... -D_{d}^{2}$ to introduce  an analog of the Paley-Wiener  spaces $PW_{\sigma}(L),\>\>\sigma>0,$  which are used as the apparatus for approximation. Since one of our auxiliary results about such spaces Theorem \ref{j-1-group} holds 
 not only for the special operator $L$ but for a general non-negative self-adjoint operator $\mathcal{L}$ in ${\bf H}$, it makes sense to define Paley-Wiener spaces for a general self-adjoint non-negative operators.
\begin{defn}\label{PW} If $\mathcal{L}$ is a non-negative self-adjoint operator in a Hilbert space $\mathbf{H}$ then 
 $PW_{\sigma} \left(\mathcal{L}\right)\subset {\bf H}$ denote the image space of the projection operator ${\bf 1}_{[0,\>\sigma]}\left(\mathcal{L}\right)$  to be understood in the sense of Borel functional calculus for self-adjoint operators. 
 \end{defn}
It is obvious that
 the space $PW_{\sigma}\left(\mathcal{L}\right)$ is a linear closed subspace in
$\mathbf{H}$ and  the space 
 $\bigcup _{ \sigma >0}PW_{\sigma}\left(\mathcal{L}\right)$
 is dense in $\mathbf{H}$.
The following theorem contains generalizations of several results
from  classical harmonic analysis (in particular  the
Paley-Wiener theorem). It follows from our  results in
\cite{Pes3}, \cite{Pes5}, \cite{Pes6}, \cite{KP}.
\begin{thm}\label{PWproprties}
The following statements hold:
\begin{enumerate}
\item (Bernstein inequality)   $f \in PW_{\sigma}\left(\mathcal{L}\right)$ if and only if
$ f \in \mathbf{H}^{\infty}=\bigcap_{k=1}^{\infty}\mathbf{H}^{k}$,
and the following Bernstein inequalities  holds true
\begin{equation}\label{Bern0}
\|\mathcal{L}^{k/2}f\|_{\mathbf{H}} \leq \sigma^{k}\|f\|_{\mathbf{H}}  \quad \mbox{for all} \, \,  k\in \mathbb{N};
\end{equation}
 \item  (Paley-Wiener theorem) $f \in PW_{\sigma}\left(\mathcal{L}\right)$
  if and only if for every $g\in\mathbf{H}$ the scalar-valued function of the real variable  $ t \mapsto
\langle e^{it\mathcal{L}}f,g \rangle $
 is bounded on the real line and has an extension to the complex
plane as an entire function of the exponential type $\sigma$.

\end{enumerate}
\end{thm}

Next, we define the best approximation
\begin{equation}\label{appr}
\mathcal{E}_{\mathcal{L}}(\sigma, f)=\inf_{g\in
PW_{\sigma}\left(\mathcal{L}\right)}\|f-g\|=\left \|f-\mathcal{P}_{\sigma}f\right\|,
\end{equation}
where $\mathcal{P}_{\sigma}$ is the orthogonal projector of $\mathbf{H}$ onto $PW_{\sigma}\left(\mathcal{L}\right)$
We also using  the Schr\"{o}dinger  group $e^{it\mathcal{L}}$ to  introduce  the modulus of continuity
\begin{equation}\label{mod-L}
\omega_{\mathcal{L}}^{r}(t,f)=\sup_{0\leq \tau\leq t}\left\|\left(e^{it\mathcal{L}}-I\right)^{r}f\right\|.
\end{equation}
In section \ref{J-S} in Theorem \ref{j-1-group} we prove  the following  Jackson-type estimate which holds for any self-adjoint operator $\mathcal{L}$
\begin{equation}\label{Schrod}
\mathcal{E}_{\mathcal{L}}(\sigma, f)\leq C(\mathcal{L})\omega_{\mathcal{L}}^{r}(\sigma^{-1}, f).
\end{equation}
Note, that (\ref{double-ineq-0}) contains   the following known  fact  (see \cite{BB})
\begin{equation}\label{K-L-ineq}
\omega_{\mathcal{L}}^{r}(s, f)\leq c_{2}K\left(s^{r},f, {\bf H}, \mathcal{D}(\mathcal{L}^{r/2})\right)\leq C_{2}\left(\omega_{\mathcal{L}}^{r}(s, f)+\min(s^{r}, 1) \|f\|\right),
\end{equation}
where $\mathcal{D}(\mathcal{L}^{r/2}) $ is  the domain  of the operator $\mathcal{L}^{r/2}$ with the graph norm $\|f||+\|\mathcal{L}^{r/2}f\|$. 

According to the assumption \textit{(d)} the graph norm of the domain $\mathcal{D}(L^{r/2})$ of the operator $L^{r/2}$ is equivalent to the norm (\ref{Sob-0}) and the spaces $\mathcal{D}(L^{r/2})$ and ${\bf H}^{r}$  coincide. It implies, in particular, existence of a constant $C_{3}>0$ such that 
 \begin{equation}\label{C3}
 K\left(s^{r},f, {\bf H}, \mathcal{D}(L^{r/2})\right)\leq C_{3}K\left(s^{r},f, {\bf H}, {\bf H}^{r}\right),\>\>\>f\in {\bf H}.
 \end{equation}
Thus in our specific situation we obtain by using (\ref{Schrod}), (\ref{K-L-ineq}), (\ref{C3}), and (\ref{double-ineq-0})
$$
\mathcal{E}_{L}(\sigma, f)\leq C(L)\omega_{L}^{r}(\sigma^{-1}, f)\leq C(L)c_{2}K\left(\sigma^{-r},f, {\bf H}, \mathcal{D}(L^{r/2})\right)\leq 
$$
$$
C(L)c_{2}C_{3}K\left(\sigma^{-r},f, {\bf H}, {\bf H}^{r}\right)\leq C(L)c_{2}C_{3}C_{1}\left(\Omega^{r}(\sigma^{-1},f)+ \min(\sigma^{-r}, 1)\|f\|_{{\bf H}}\right)
$$
Now we can formulate our main theorem.

\begin{thm}\label{main-thm}If the assumptions \textit{(a)-(d)} are satisfied then there exists a constant $C>0$ which is independent on $f\in {\bf H}$ such that
\begin{equation}\label{main}
\mathcal{E}_{L}(\sigma, f)\leq C\left(\Omega^{r}( \sigma^{-1}, f)+\min (\sigma^{-r}, 1)\|f\|\right),
\end{equation}
where $\mathcal{E}_{L}(\sigma, f)$ and $\Omega^{r}( \sigma^{-1}, f)$ defined in (\ref{appr}) and (\ref{Mod-0}) respectively.
\end{thm}

\begin{remark}
It is important to notice that since $\Omega^{r}(f, \tau)$ cannot be of order $o(\tau^{r})$ when $\tau\rightarrow 0$ (unless $f$ is invariant), the behavior of the right-hand side in (\ref{main}) is determined by the first term when $\sigma\rightarrow \infty$. In particular, if $f\in {\bf H}^{r}$, then due to the inequality
$$
\Omega^{r}(s, f)\leq s^{-k}\Omega^{r-k}(s, D_{j_{1}}...D_{j_{k}}f),\>\>\>\>\>0\leq k\leq r,
$$
one has the best possible estimate
$$
\mathcal{E}_{L}(\sigma, f)\leq C\Omega^{r}(\sigma, f)\leq C\sigma^{-r}\|f\|_{{\bf H}^{r}}.
$$

\end{remark}

\section{  Jackson inequality for the  Schr\"{o}dinger group of a self-adjoint operator }\label{J-S}

The next lemma shows that the inequality (\ref{Bern0}) can be relaxed. This fact is used in the proof of  Lemma \ref{Lem2} which suggests  a way of construction of Paley-Wiener vectors.

\begin{lem}\label{Lem1}
If  $\mathcal{L}$ is a self-adjoint operator in a Hilbert space ${\bf H}$ then a vector $f$ belongs to 
the subspace $PW_{\sigma}(\mathcal{L}),\>\sigma>0, $ if and only if there exists a constant $C=C(f,\sigma)>0$ such that for all $k\in \mathbb{N}$ 
\begin{equation}\label{C-ineq}
\|\mathcal{L}^{k/2}f\|\leq C\sigma^{k}\|f\|.
\end{equation}

\end{lem}

\begin{proof} If $f\in PW_{\sigma}(\mathcal{L})$ then by (\ref{Bern0}) the inequality (\ref{C-ineq}) holds with $C=1$. Conversely, 
if  for an $f\in {\bf H}$  the inequality (\ref{C-ineq}) holds  for some $C=C(f,\sigma)$ then  for
any complex number $z$ we have

$$ 
\left\|e^{z\mathcal{L}}f\right\|=\left\|\sum ^{\infty}_{m=0}(z^{m}\mathcal{L}^{m}f)/m!\right\|\leq C \sum
^{\infty}_{m=0}|z|^{m}\sigma^{m}/m!=Ce^{|z|\sigma}.
$$
It implies that for any functional $\psi^{*}\in E^{*}$ the scalar function
$\left<e^{z\mathcal{L}}f,\psi^{*}\right>$ is an entire function
 of exponential type $\sigma $ which is bounded on the real axis 
by the constant $\|\psi^{*}\| \|f\|$.
 An application of the classical Bernstein inequality gives

$$\left\|\left<e^{t\mathcal{L}}\mathcal{L}^{k}f,\psi^{*}\right>\right\|_{C(R^{1})}=\left\|\left(\frac{d}{dt}\right)^{k}\left<e^{t\mathcal{L}}f,\psi^{*}\right>\right\|_{
C(R^{1})} \leq\sigma^{k}\|\psi^{*}\| \|f\|.$$
From here  for $t=0$ we obtain 
$$ \left|\left<\mathcal{L}^{k}f,\psi^{*}\right>\right|\leq \sigma ^{k} \|\psi^{*}\| \|f\|.
$$
Choice of $\psi^{*}\in E^{*}$ such that $\|\psi^{*}\|=1$ and $\left<\mathcal{L}^{k}f,\psi^{*}\right>=\|\mathcal{L}^{k}f\|$ gives
the inequality $\|\mathcal{L}^{k}f\|\leq
 \sigma ^{k} \|f\|,\>\> k\in \mathbb{N}$, which implies Theorem.
\end{proof}

\begin{lem}\label{Lem2}
If  $\mathcal{L}$ is a self-adjoint operator in a Hilbert space ${\bf H}$ and  $p\in L_{1}(\mathbb{R})$ is an entire function of exponential
type $\sigma$ then for any $f\in {\bf H}$ the vector
$$
P_{\sigma}f=\int _{-\infty}^{\infty}p(t)e^{it\mathcal{L}}fdt
$$
belongs to $PW_{\sigma}(\mathcal{L}).$

\end{lem}
\begin{proof}

For $g=P_{\sigma}f,\>f\in {\bf H},$ and for every real $\tau$ we
have
$$
e^{i\tau
\mathcal{L}}g=\int_{-\infty}^{\infty}p(t)e^{i(t+\tau)\mathcal{L}}fdt=\int_{-\infty}^{\infty}p(
t-\tau)e^{it\mathcal{L}}fdt.
$$
 Using this formula we can extend the abstract function $e^{i\tau \mathcal{L}}g$ to the
complex plane as
$$
e^{iz\mathcal{L}}g=\int_{-\infty}^{\infty}p(t-z)e^{it\mathcal{L}}fdt.
$$
One has 
$$
\|e^{iz\mathcal{L}}g\|\leq
\|f\|\int_{-\infty}^{\infty}|p(t-z)|dt.
 $$
 Since by assumption $p\in L_{1}(\mathbb{R})$ is an entire function of exponential
type $\sigma$
 we have for $z=x+iy$ and $u=t-x$
 $$
\int_{-\infty}^{\infty}|p(t-z)|dt= \int_{-\infty}^{\infty}|p(u-iy)|du\leq e^{\sigma |y|}\|p\|_{L_{1}(\mathbb{R})},
 $$
 because
  $$
  p(u-iy)=\sum_{m=0}^{\infty}  \frac{(-iy)^{m}}{m!}p^{(m)}(u),
  $$
  and according to the classical Bernstein inequality
  $$
  \>\>\>\>\|p^{(m)}\|_{L_{1}(\mathbb{R})}\leq \sigma^{m} \|p\|_{L_{1}(\mathbb{R})}.
  $$
  Thus
$$
\|e^{iz\mathcal{L}}g\|\leq
\|f\|\int_{-\infty}^{\infty}|p(t-z)|dt\leq\|f\|e^{\sigma |y|}
\|h\|_{L_{1}}.
 $$
It shows that for every vector $h\in {\bf H}$ the function
$\left<e^{iz\mathcal{L}}g,h\right>$ is an entire function and
$$
\left|\left<e^{iz\mathcal{L}}g,h\right>\right|\leq
\|h\|\|f\|e^{\sigma |y|}\|f\|_{L_{1}(\mathbb{R})}.
$$
In other words the $\left<e^{iz\mathcal{L}}g,h\right>$ is an entire
function of the exponential type $\sigma$
  which is bounded on the real line and another application of the classical
Bernstein inequality in the norm $C(\mathbb{R})$  gives the inequality
$$
\left|\left(\frac{d}{dt}\right)^{k}\left<e^{it\mathcal{L}}g,h\right>\right|
\leq\sigma^{k}\sup_{t\in\mathbb{R}}\left|\left<e^{it\mathcal{L}}g,h\right>\right|.
$$
Since
$$
\left(\frac{d}{dt}\right)^{k}\left<e^{it\mathcal{L}}g,h\right>=\left<e^{it\mathcal{L}}(i\mathcal{L})^{k}g,h\right>
$$
we obtain for $t=0$
$$
\left|\left<\mathcal{L}^{k}g,h\right>\right|\leq
\sigma^{k}\|h\|\|f\|\int_{-\infty}^{\infty}|p(\tau)| d\tau.
$$
 Choosing $h$ such that $\|h\|=1$ and $\left<\mathcal{L}^{k}g,h\right>=\|\mathcal{L}^{k}g\|$
  we obtain the inequality
$$
\|\mathcal{L}^{k}g\|\leq
\sigma^{k}\|f\|\int_{-\infty}^{\infty}|p(\tau)|d\tau,\>\>\>k\in \mathbb{N},
$$
which  according to Lemma \ref{Lem1}  implies that $g$ belongs to $PW_{\sigma}(\mathcal{L})$. Lemma is proven.

\end{proof}

For the modulus of continuity introduced in (\ref{mod-L}) the following inequalities hold:
\begin{equation}
\omega_{\mathcal{L}}^{m}\left( s, f\right)\leq s^{k}\omega^{m-k}_{\mathcal{L}}(s, \mathcal{L}^{k}f)\label{first},\>\>\>\>0\leq k\leq m,
\end{equation}
and
\begin{equation}
\omega_{\mathcal{L}}^{m}\left(as, f\right)\leq \left(1+a\right)^{m}\omega_{\mathcal{L}}^{m}(s, f), \>\>\>a\in \mathbb{R}_{+}.\label{second}
\end{equation}
The first one follows from the identity
\begin{equation}
\left(e^{is \mathcal{L}}-I\right)^{k}f=\int_{0}^{s}...\int_{0}^{s}e^{i(\tau_{1}+...\tau_{k})\mathcal{L}}
\mathcal{L}^{k}fd\tau_{1}...d\tau_{k},\label{id3}
\end{equation}
where $I$ is the identity operator and $k\in \mathbb{N}$.  The
second one follows from the property
$$
\omega_{\mathcal{L}}^{1}\left( s_{1}+s_{2}, f\right)\leq\omega_{\mathcal{L}}^{1}\left(
s_{1}, f\right)+\omega_{\mathcal{L}}^{1}\left(s_{2}, f\right)
$$
which is easy to verify. 
The next Theorem and its proof are motivated by  Theorem 5.2.1  in \cite{N}.

\begin{thm}\label{j-1-group} Let  $\mathcal{L}$ be a self-adjoint operator in a Hilbert space ${\bf H}$. For a given natural $m$ there exists a constant $c=c(m)>0$ such that for all
$\sigma>0$ and all $f$ in $\mathbf{H}$ 
\begin{equation}\label{100}
\mathcal{E}_{\mathcal{L}}(\sigma, f)\leq
c\omega^{m}_{\mathcal{L}}\left( 1/\sigma, f\right).
\end{equation}
Moreover, for any $1\leq k\leq m$ there exists a $C=C(m,k)>0$ such that for any  $f\in \mathcal{D}(\mathcal{L}^{k})$ one has
\begin{equation}
\mathcal{E}_{\mathcal{L}}(\sigma, f)\leq
\frac{C}{\sigma^{k}}\omega^{m-k}_{\mathcal{L}}\left( 1/\sigma, \mathcal{L}^{k}f\right),\>\>\>\>
0\leq k\leq m.\label{200}
\end{equation}
\end{thm}
\begin{proof}

Let
\begin{equation}
\rho(t)=a\left(\frac{\sin (t/n)}{t}\right)^{n}
\end{equation}
where $n=2( m+3)$  and
$$
a=\left(\int_{-\infty}^{\infty}\left(\frac{\sin
(t/n)}{t}\right)^{n}dt\right)^{-1}.
$$
With such choice of $a$ and $n$ function $\rho$ will have the
 following properties:

(1) $\rho$ is an even nonnegative entire function of exponential
type one;

(2) $\rho$ belongs to $L_{1}(\mathbb{R})$ and its
$L_{1}(\mathbb{R})$-norm is $1$;

(3)  the integral
\begin{equation}
\int_{-\infty}^{\infty}\rho(t)|t|^{m}dt
\end{equation} 
is finite.

Next, we observe the following formula
\begin{equation}\label{binom}
(-1)^{m+1}(e^{si\mathcal{L}}-I)^{m}f=
$$
$$
(-1)^{m+1}\sum^{m}_{j=0}(-1)^{m-j}C^{j}_{m}e^{js(i\mathcal{L})}f=
 \sum_{j=1}^{m}b_{j}e^{js(i\mathcal{L})}f-f,
 \end{equation}
 where $b_{1}+b_{2}+ ... +b_{m}=1.$
Consider
 the vector 
\begin{equation}\label{id}
\mathcal{Q}_{\rho}^{\sigma,m}(f)=\int_{-\infty}^{\infty}
\rho(t)\left\{(-1)^{m+1}(e^{\frac{t}{\sigma}i\mathcal{L}}-I)^{m}f+f\right\}dt.
 \end{equation}
 According to (\ref{binom}) we have 
 $$
\mathcal{Q}_{\rho}^{\sigma,m}(f) =                  \int_{-\infty}^{\infty}
\rho(t)\sum_{j=1}^{m}b_{j}e^{j\frac{t}{\sigma}(i\mathcal{L})}fdt.
$$
Changing variables in each of integrals 
$$
\int_{-\infty}^{\infty}\rho(t)e^{j\frac{t}{\sigma}i\mathcal{L}}fdt,\>\>\>\>1\leq j\leq m,
$$
we obtain the formula
$$
\mathcal{Q}_{\rho}^{\sigma,m}(f) =\int_{-\infty}^{\infty}\Phi(t)e^{t(i\mathcal{L})}fdt,
$$
where
$$
\Phi(t)=\sum_{j=1}^{m}b_{j}\left(\frac{\sigma}{j}\right)\rho\left(t\frac{\sigma}{j}\right), \>\>\>\>\>\>b_{1}+b_{2}+...
+b_{m}=1.
 $$ 
Since the function $\rho(t)$ has exponential type one every function
$\rho(t\sigma/j)$ has the type $\sigma/j$ and because of  this 
the function $\Phi(t)$ 
has exponential  type $\sigma$. It  also belongs to $L_{1}(\mathbb{R})$ and as it was just shown it implies that the vector $\mathcal{Q}_{\rho}^{\sigma,m}(f) $ belongs to $PW_{\sigma}(\mathcal{L})$.
Now we estimate the error of approximation of
$\mathcal{Q}_{\rho}^{\sigma, m}(f)$  to $f$. 
 Since by (\ref{id})
$$
\mathcal{Q}_{\rho}^{\sigma,m}(f)-f=
(-1)^{m+1}\int_{-\infty}^{\infty}\rho(t)(e^{\frac{t}{\sigma}i\mathcal{L}}-I)^{m}fdt
$$
we obtain by using (\ref{second}) 
$$
\mathcal{E}_{\mathcal{L}}(\sigma, f)\leq\|f-\mathcal{Q}_{\rho}^{\sigma,m}(f)\|\leq
\int_{-\infty}^{\infty}\rho(t)\left    \|         (e^{\frac{t}{\sigma}i\mathcal{L}}-I)^{m}f\right\|     dt\leq
$$
$$
\int_{-\infty}^{\infty}\rho(t)\omega^{m}_{\mathcal{L}}\left( t/\sigma,\>f\right)dt \leq c\omega^{m}_{\mathcal{L}}\left(1/\sigma,\>f\right), \>\>\>\>\>\>\>c=\int_{-\infty}^{\infty}\rho(t)(1+|t|)^{m}dt.
$$
If $f\in \mathcal{D}(\mathcal{L}^{k})$ then by using (\ref{first}) 
we have
$$
\mathcal{E}_{\mathcal{L}}(\sigma, f)\leq
\int_{-\infty}^{\infty}\rho(t)\omega^{m}_{\mathcal{L}}\left(t/\sigma, f\right)dt
\leq 
$$
$$
\frac{\omega^{m-k}_{\mathcal{L}}\left(
1/\sigma, \mathcal{L}^{k}f\right)}{\sigma^{k}}\int_{-\infty}^{\infty}\rho(t)|t|^{k}(1+|t|)^{m-k}dt\leq
\frac{{C}}{\sigma^{k}}\omega^{m-k}_{\mathcal{L}}\left(
1/\omega, \mathcal{L}^{k}f\right),
$$
where
 $$
C=\int_{-\infty}^{\infty}\rho(t)|t|^{k}(1+|t|)^{m-k}dt,
$$
 is finite by the choice of $\rho$. The inequalities (\ref{100}) and (\ref{200}) are proved.
\end{proof}

\section{Unitary representations of Lie groups}

A strongly continuous unitary representation of a Lie group $G$ in a Hilbert space $\mathbf{H} $ is a homomorphism  $T: G \mapsto U(\mathbf{H})$ where $U(\mathbf{E})$ is the group of unitary operators of $\mathbf{H}$ such that $T(g)f,\>\>g\in G,$ is continuous on $G$ for any $f\in \mathbf{H}$.  The Garding space $\mathcal{G}$ is defined as the set of vectors $h$  in $\mathbf{H}$ that have the representation
$
h=\int_{G}\varphi(g)T(g)f dg,
$
where $f\in \mathbf{H}$, $\>\>\varphi\in C_{0}^{\infty}(G)$, $\>\>dg$ is a left-invariant measure on $G$. Every element  $X$ which belongs to the corresponding Lie algebra $ \mathbf{g}$ can be  identified with a  right-invariant vector field
$$
X\varphi(g)=\lim_{t\rightarrow 0}\frac{\varphi\left( \exp tX \cdot g\right)-\varphi(g)}{t}.
$$ 
The correspondence $X\rightarrow D(X)$ is a representation  of $\mathbf{g}$ by operators which act  on $\mathcal{G}$ by the formula 
$$
D(X)h=-\int_{G}X\varphi(g)T(g)fdg.
$$

 If $X_{1}, ..., X_{d}$ is a basis in $\mathbf{g}$ and $D_{i}=D(X_{i}),\>\>1\leq i \leq d,$
we introduce an analog of a Sobolev space  as the subspace $\mathbf{H}^{r}\subset {\bf H}$ which is the common domain of all the operators $D_{j_{1}}...D_{j_{k}},\>\>1\leq j_{1},  ...j_{k}\leq d,\>\>1\leq k\leq r, $ with the norm 
 
$$
|||f|||_{\mathbf{H}^{r}}=\|f\|_{\mathbf{H}}+\sum_{k=1}^{r}\sum_{1\leq j_{1},  ...j_{k}\leq d}\|D_{j_{1}}...D_{j_{k}}f\|_{\mathbf{H}},
$$
which is equivalent to the norm 
\begin{equation}\label{Sob-01}
\|f\|_{r}=\|f\|_{\mathbf{H}}+\sum_{1\leq i_{1},..., i_{r}\leq
d}\|D_{i_{1}}...D_{i_{k}}f\|_{\mathbf{H}},\>\> ~ r\in \mathbb{N}.
\end{equation}
It is known  that 
$
\mathcal{G}\subset \bigcap_{r\in N}\mathbf{H}^{r}=\mathbf{H}^{\infty}
$
 is invariant with respect  to all operators $D(X), \>\>X\in \textbf{g},$ and dense in every $\mathbf{H}^{r}$.
We consider the following analog of the Laplace operator \cite{Nel}, \cite{NS}
 $$
 L_{\mathcal{G}}=-D_{1}^{2}-D_{2}^{2}- ... -D_{d}^{2},
 $$
 which is  defined on the Garding space $\mathcal{G}$. 
 Since $L_{\mathcal{G}}$ is symmetric and the differential operator $-\sum_{i=1}^{d}X_{i}^{2}$ is elliptic  on the group $G$ the Theorem 2.2 in \cite{NS} implies that $L_{\mathcal{G}}$ is essentially self-adjoint, which means $\overline{L}_{\mathcal{G}}=L_{\mathcal{G}}^{*}$. In other words, the closure $\overline{L}_{\mathcal{G}}=L$ of $L_{\mathcal{G}}$ from $\mathcal{G}$ is a self-adjoint operator. Obviously, $L\geq 0$. We introduce the self-adjoint operator 
 $
 \Lambda=I+L\geq 0.
 $
 
 It was shown in \cite{Pes8}, Lemma 2.1 and Theorem 2.2, (see also \cite{Pes4},\cite{Pes5})  that in the case of a unitary representation $T$ of a Lie group $G,\>dim\>G=d,$ all our {\bf Assumptions (a)-(d)} are satisfied for one-parameter groups $T_{1}(t), ..., T_{d}(t)$, where for a basis $X_{1}, .., X_{d}$, of the Lie algebra of $G$ one has
 $$
 T_{j}(t)=T(\exp tX_{j}),\>\>\>\>t\in \mathbb{R},\>\>\>1\leq j\leq d,
 $$
 and where $\exp tX_{j}\in G,\>t\in \mathbb{R},$  is the one parameter subgroup in the direction of $X_{j}$.
 In particular, a proof of the following important fact is given in Appendix \ref{A}.

\begin{thm}\label{domain}
The space $\mathbf{H}^{r}$ with the norm (\ref{Sob-01}) is isomorphic  to the domain of $\Lambda^{r/2}$ with the norm $\|\Lambda^{r/2}f\|_{\mathbf{H}}.$
\end{thm}

An important class of representations of Lie groups appears in connection with homogeneous manifolds.
In what follows we introduce some very basic notions about unitary representations of Lie groups in function spaces on homogeneous manifolds \cite{H}, 
 \cite{V}.

 Let $M, dim M=m,$ be a
connected $C^{\infty}$-manifold. It says  that a 
Lie group $G$ effectively acts on $M$ as a group of
diffeomorphisms if

1)  every element $g\in G$ can be identified with a diffeomorphism
$$
g: M\mapsto M
$$
of $M$ onto itself and
$$
g_{1}g_{2}\cdot x=g_{1}\cdot(g_{2}\cdot x), g_{1}, g_{2}\in G,
x\in M,
$$
where $g_{1}g_{2}$ is the product in $G$ and $g\cdot x$ is the
image of $x$ under $g$,

2) the identity $e\in G$ corresponds to the trivial diffeomorphism
\begin{equation}
e\cdot x=x,
\end{equation}

3) for every $g\in G, g\neq e,$ there exists a point $x\in M$ such
that $g\cdot x\neq x$.

\bigskip

A group $G$ acts on $M$ \textit{transitively} if in addition to
1)- 3) the following property holds

4) for any two points $x,y\in M$ there exists a diffeomorphism
$g\in G$ such that
$$
g\cdot x=y.
$$

A \textit{homogeneous} manifold $M$ is an
$C^{\infty}$-compact manifold on which transitively acts a 
Lie group $G$. In this case $M$ is necessary of the form $G/K$,
where $K$ is a subgroup of $G$. The notation $L_{2}(M)$ is used for the usual Hilbert spaces
$L_{2}(M,dx)$, where $dx$ is an invariant (with respect to $G$-action)
measure.
It is known that the correspondence 
$$
g\rightarrow T(g),\>\>\>\>\>T(g)f(x)=f(g\cdot x), 
$$
 where $g\in G,\>\>x\in M,\>\>f\in L_{2}(M)$ is a unitary representation of $G$ in $L_{2}(M)$.

\begin{exmp} {\bf A compact homogeneous manifold.}
The situation on a  unit sphere is typical for at least all two-point homogeneous compact manifolds.
Consider the unit sphere
$$
\mathbb{S}^{n}=\left\{x=(x_{1}, x_{2}, ...,x_{n+1})\in \mathbb{R}^{n+1}: \|x\|^{2}=x_{1}^{2}+x_{2}^{2}+...+x_{n+1}^{2}=1\right\}.
$$
Let $e_{1},...,e_{n+1}$ be the standard orthonormal basis in $\mathbb{R}^{n+1}$.  If $SO(n+1)$ and $SO(n)$ are the groups of rotations of $\mathbb{R}^{n+1}$ and  $\mathbb{R}^{n}$ respectively then $\mathbb{S}^{n}=SO(n+1)/SO(n)$.
On $\mathbb{S}^{n}$ we consider vector fields
$
X_{i,j}=x_{j}\partial_{x_{i}}-x_{i}\partial_{x_{j}},\>i<j\>,
$
which are generators of one-parameter groups of rotations   $\exp tX_{i,j}\in SO(n+1)$ in the plane $(x_{i}, x_{j})$.  These groups are defined by the formulas for $\tau\in \mathbb{R}$,
$$
\exp \tau X_{i,j}\cdot (x_{1},...,x_{n+1})=(x_{1},...,x_{i}\cos \tau -x_{j}\sin \tau ,..., x_{i}\sin \tau +x_{j}\cos \tau ,..., x_{n+1}).
$$
Clearly, there are $d=\frac{1}{2}n(n-1)$ such groups.
Let $T_{i,j}(\tau)$ be a one-parameter group which is a representation of $\exp \tau X_{i,j}$ in the space $L_{2}(\mathbb{S}^{})$. It acts on $f\in L_{2}(\mathbb{S}^{n})$ by the following formula
$$
T_{i,j}(\tau)f(x_{1},...,x_{n+1})=f(x_{1},...,x_{i}\cos \tau -x_{j}\sin \tau ,..., x_{i}\sin \tau +x_{j}\cos \tau ,..., x_{n+1}).
$$
The infinitesimal operator of this group will be denoted as $D_{i,j}$. The operator $L=-\sum_{i<j}D_{i,j}^{2}$ is the regular Laplace-Beltrami operator on $
\mathbb{S}^{n}$ and spaces $PW_{\sigma}(\mathbb{S}^{n})$ are comprised of appropriate  linear combinations of spherical harmonics. 

\begin{remark}This example explains  reasons why $d$ is typically greater than $n=dim\>M$. In this case it happens  because vector fields $D_{j}$ can  vanish along low dimensional submanifolds.  For example, on $\mathbb{S}^{2}\subset \mathbb{R}^{3}$ one needs three fields $X_{1, 2}, \>X_{1, 3}, \>X_{2, 3}$ since they vanish at the poles $(0, 0, \pm1),\> (0, \pm1, 0), \> (\pm1, 0,0)$ respectively.
\end{remark}

\end{exmp}

\begin{exmp} {\bf A non-compact homogeneous manifold.}

Consider the upper half of the hyperboloid
$$
\mathbb{H}_{+}^{n} = \{x=(x_{1}, x_{2}, ...,x_{n+1})\in \mathbb{R}^{n+1}:  -x_{1}^{2}-x_{2}^{2}-...-x_{n}^{2}+x_{n+1}^{2}=1,\>x_{n+1}>0              \}.
$$
Let $e_{1},...,e_{n+1}$ be the standard orthonormal basis in $\mathbb{R}^{n+1}$.  If $SH(n+1)$ is the group of hyperbolic rotations which means it preserves the form
$$
[x,y]=-x_{1}y_{1}-...-x_{n}y_{n}+x_{n+1}y_{n+1}
$$
then $\mathbb{H}_{+}^{n}=SH(n+1)/SO(n)$.
On $\mathbb{H}_{+}^{n}$ we consider the vector fields
$
X_{i,j}=x_{j}\partial_{x_{i}}-x_{i}\partial_{x_{j}},\>i<j<n+1\>,
$
which generate euclidean rotation groups in  the planes $(x_{i}, x_{j}),\>i<j<n+1$,
and the fields $X_{i, n+1}=x_{n+1}\partial_{x_{i}}+x_{i}\partial_{x_{n+1}}$
which are generators of the hyperbolic  groups of rotations    in the planes $(x_{i}, x_{n+1})$.  These groups are defined by the formulas for $\tau\in \mathbb{R}$,
$$
\exp \tau X_{i,j}\cdot (x_{1},...,x_{n+1})=(x_{1},...,x_{i}\cos \tau -x_{j}\sin \tau ,..., x_{i}\sin \tau +x_{j}\cos \tau, ..., x_{n+1} ),
$$
$$
\exp \tau X_{i,n+1}\cdot (x_{1},...,x_{n+1})=(x_{1},...,x_{i}\cosh \tau -x_{n+1}\sinh \tau ,..., x_{i}\sinh \tau +x_{n+1}\cosh \tau ).
$$
Strictly continuous one-parameter groups of operators $T_{i,j}(\tau)$   which are representations of $\exp \tau X_{i,j}$ in the space $L_{2}(\mathbb{H}^{n}_{+})$ can be used to construct corresponding modulus of continuity $\Omega^{r}(\sigma, f)$.  Their   infinitesimal operators $D_{i,j}$ are just operators $X_{i,j}$ in the space $L_{2}(\mathbb{H}^{n}_{+})$ and  $L=-\sum_{i<j\leq n+1}D_{i,j}^{2}$ is an elliptic self-adjoint non-negative operator   in  $
L_{2}(\mathbb{H}^{n}_{+})$ which has continuous spectrum.   As well as we know, the spectral resolution of this operator is unknown. However, the abstract Definition \ref{PW} and notion of best approximation (\ref{appr}) still make sense.  
\end{exmp}

\begin{exmp} {\bf Schr\"{o}dinger representation of the Heisenberg group.} 

The $(2n+1)$-dimensional Heisenberg group $\mathbb{H}_{2n+1}$ has a unitary representation in the space $L_{2}(\mathbb{R}^{n})$ 
$$
T(p,\>q,\>x)=e^{i(t+\left<q,\>x\right>)}f(x+p), \>\>\>p,\>q,\>x\in \mathbb{R}^{n},\>\>t\in \mathbb{R}. 
$$
One can consider  the following set of infinitesimal operators where $i=\sqrt{-1}$:
$$
D_{j}=\partial_{j}, \>\>1\leq j\leq n;\>\>\>\>D_{j}=ix_{j}, \>\>n+1\leq j\leq 2n;\>\>\>\> D_{2n+1}=i.
$$
In this case every $T_{j}(\tau) , \>1\leq j\leq n,\>$ is a translation (\ref{transl}) along variable $x_{j}$,
$$
T_{j}(\tau)f(x_{1}, ..., x_{n})=e^{i\tau x_{j}}f(x_{1}, ..., x_{n}), \>\>\>n+1\leq j\leq 2n,\>
$$
and $T_{2n+1}(\tau)f(x_{1}, ..., x_{n})=e^{i\tau}f(x_{1}, ..., x_{n})$.
The operator $L$ is the shifted $n$-dimensional linear oscillator 
$$
L=-\Delta+|x|^{2}+1,\>\>\>
$$
where 
$$
\Delta=\sum_{j=1}^{n}\partial_{j}^{2},\>\>\>\>|x|^{2}=\sum_{j}x_{j}^{2},\>\>\>\>x=(x_{1},...,x_{n}).
$$
It is known that the spectrum of this operator is discrete and its eigenfunctions are products of one-dimensional Hermite functions. One can easily describe corresponding Paley-Wiener spaces and to construct corresponding modulus of continuity by using groups of operators $T_{j},\>\>1\leq j\leq 2n+1$.

\end{exmp}

\section{Appendix. Proof of Theorem \ref{domain}}\label{A}

\begin{thm}
The space $\mathbf{H}^{r}$ with the norm (\ref{Sob-01}) is isomorphic  to the domain of $\Lambda^{r/2}$ with the norm $\|\Lambda^{r/2}f\|_{\mathbf{H}}.$
\end{thm}

\begin{proof}
In the case $r=2k,$ the inequality 
 \begin{equation}\label{2.9}
  \|f\|_{\mathbf{H}^{2k}}\leq C(k)\|\Lambda^{k}f\|_{\mathbf{H}}
\end{equation}
is shown in \cite{Nel}, Lemma 6.3. The reverse inequality is obvious.   
We consider now the case $r=2k+1$.  If $f\in \mathbf{H}^{2}=\mathcal{D}(\Lambda)$, then since 
$
\mathcal{D}(\Lambda)\subset \mathcal{D}(\Lambda^{1/2})
$
we have
\begin{equation}
\label{1/2}
\|f\|_{\mathbf{H}}
^{2}+\sum_{j}\|D_{j}f\|_{\mathbf{H}}
^{2}=\left<f,f\right>+\sum_{j}\left<D_{j}f,D_{j}f\right>=\left<f,f\right>+\left<-\sum_{j}D_{j}^{2}f,f\right>=
$$
$$
\left<f-\sum_{j}D_{j}^{2}f,f\right>=\left<\Lambda f,f\right>=\|\Lambda^{1/2}f\|_{\mathbf{H}}
^{2}.
\end{equation}
These equalities imply that $\mathbf{H}^{1}$ is isomorphic to $\mathcal{D}(\Lambda^{1/2})$.
Our goal is to to prove existence of an isomorphism between $\mathbf{H}^{2k+1}$ and $\mathcal{D}(\Lambda^{k+1/2})$. It is enough to establish equivalence of the corresponding norms on the set $\mathbf{H}^{4k+2}=\mathcal{D}(\Lambda^{2k+1})$ since the latest is dense in $\mathbf{H}^{2k+1}$.
If $f\in \mathbf{H}^{4k+2}\subset \mathbf{H}^{2k}$ then $D_{j}f\in \mathbf{H}^{4k+1}\subset \mathbf{H}^{2k}$ and 
$
\Lambda^{k}f=\sum_{m\leq k}\sum D_{j_{1}}^{2}...D_{j_{m}}^{2}f.
$
Thus if $f\in \mathbf{H}^{4k+2}$ then
$$
\left\|D_{j_{1}}...D_{j_{2k+1}}f  \right \|_{\mathbf{H}}
\leq C\left\|\Lambda^{k}D_{j_{2k+1}}f \right\|_{\mathbf{H}}
=\left\|\sum_{m\leq k}\sum D_{j_{1}}^{2}...D_{j_{m}}^{2}D_{j_{2k+1}}f \right\|_{\mathbf{H}}
.
$$
Multiple applications of the identity $D_{i}D_{j}-D_{j}D_{i}=\sum_{k} c_{i,j}^{k}D_{k}$ which holds on $\mathbf{H}^{2}$  lead to the inequality
$
\left\|D_{j_{1}}...D_{j_{2k+1}}f \right \|_{\mathbf{H}}
 \leq C\left(\|D_{j_{2k+1}}\Lambda^{k}f\|_{\mathbf{H}}
+\|Rf\|_{\mathbf{H}}
\right),
$
where $R$ is a polynomial in $D_{1},...,D_{d}$ whose degree $\leq 2k$. According to (\ref{2.9}) and (\ref{1/2}) we have that 
$$
\left\|D_{j_{2k+1}}\Lambda^{k}f \right\|_{\mathbf{H}}
 \leq  \left\|\Lambda^{1/2}\Lambda^{k}f \right\|_{\mathbf{H}}
=\left\|\Lambda^{k+1/2}f \right\|_{\mathbf{H}}
$$
and also
$
\left\|Rf \right \|_{\mathbf{H}}
\leq \|f\|_{\mathbf{H}^{2k}}\leq C(k)\left\|\Lambda^{k}f \right\|_{\mathbf{H}}
.
$
Since $\|\Lambda^{k}f\|_{\mathbf{H}}
$ is not decreasing with $k$ we get the following estimate 
$$
\|D_{j_{1}}...D_{j_{2k+1}}f\|_{\mathbf{H}}
\leq C(k)\|\Lambda^{k+1/2}f\|_{\mathbf{H}}
, \>\>\>f\in \mathbf{H}^{4k+2}.
$$
Now, since for $f\in \mathbf{H}^{4k+2}$ we have 
$
D_{j_{1}}... D_{j_{2k}}f\in \mathbf{H}^{2k+2}\subset \mathbf{H}^{1}=\mathcal{D}(\Lambda^{1/2}),
$
and the equality 
$
\Lambda^{k}f=\sum_{m\leq k}\sum D_{j_{1}}^{2}...D_{j_{m}}^{2}f,
$
holds we obtain, by using (\ref{1/2}) 
$$
\|\Lambda^{k+1/2}f\|_{\mathbf{H}}
=\|\Lambda^{1/2}\sum_{m\leq k}\sum D_{j_{1}}^{2}...D_{j_{m}}^{2}f\|_{\mathbf{H}}
\leq C\|f\|_{\mathbf{H}^{2k+1}},\>\>\>\>C=C(k).
$$
Theorem is proved.  
\end{proof}

\begin{col}
If $T$ is a strongly continuous unitary representation of a Lie group in a Hilbert space $\mathbf{H}$ and $X_{1},  ..., X_{d}$ is a basis in the corresponding algebra Lie $\mathbf{g}$ then for $T_{j}(t)=T(\exp tX_{j}),\>\>1\leq j\leq d,$ and their generators $D_{j}, \>\>1\leq j\leq d,$ the item (d) in the {\bf Assumptions}  is satisfied.
\end{col}

 \makeatletter
\renewcommand{\@biblabel}[1]{\hfill#1.}\makeatother

\end{document}